\documentclass[a4paper,10pt]{amsart}

\usepackage{mathtools,amssymb}
\usepackage{amsthm}

\usepackage[utf8]{inputenc}
\usepackage[T1]{fontenc}
\usepackage{newpxtext,newpxmath}

\usepackage{xcolor}
\usepackage{hyperref}
\usepackage{verbatim}
\usepackage{enumitem}
\theoremstyle{plain}
\newtheorem{thm}{Theorem}[section]
\newtheorem{lem}[thm]{Lemma}
\newtheorem{prop}[thm]{Proposition}
\newtheorem{cor}[thm]{Corollary}

\theoremstyle{definition}
\newtheorem{defn}[thm]{Definition}

\theoremstyle{remark}
\newtheorem{rem}[thm]{Remark}

\newcommand{\eps}{\varepsilon}
\newcommand{\Implies}[2]{$\text{\ref{#1}}\implies\text{\ref{#2}}$}% X => Y
\newcommand{\IFF}[2]{$\text{\ref{#1}}\Longleftrightarrow\text{\ref{#2}}$}% X => Y
\newcommand{\Fbar}{\bar{E}}

\newcommand{\V}{\hat{\omega}_T}
\newcommand{\M}{\mathcal{M}}
\newcommand{\MT}{\mathcal{M}_T}
\newcommand{\MTe}{\mathcal{M}_T^e}

\newcommand{\m}{m_T}

\newcommand{\dd}{\mathop{}\!\mathrm{d}}

\DeclareMathOperator{\diam}{diam}
\DeclareMathOperator{\orb}{Orb}

\begin{document}
\title[On the properties of the mean orbital pseudo-metric]{On the properties of the mean orbital pseudo-metric}
\author{Fangzhou Cai, Dominik Kwietniak, Jian Li, Habibeh Pourmand}

\address[F. Cai]{Department of Mathematics, Shantou University,
	Shantou, Guangdong 515063, P.R. China}
\email{cfz@mail.ustc.edu.cn}

\address[D. Kwietniak]{Faculty of Mathematics and Computer Science
Jagiellonian University in Krak\'ow
ul. {\L}ojasiewicza 6, 30-348 Kraków, Poland}
\email{dominik.kwietniak@uj.edu.pl}

\address[J. Li]{Department of Mathematics, Shantou University,
	Shantou, Guangdong 515063, P.R. China}
\email{lijian09@mail.ustc.edu.cn}

\address[H. Pourmand]{Faculty of Mathematics and Computer Science
Jagiellonian University in Kraków
ul. {\L}ojasiewicza 6, 30-348 Krak\'ow, Poland}
\email{habibeh.pourmand@im.uj.edu.pl}

\date{\today}

\subjclass[2010]{37B05, 37A25}

\keywords{The mean orbital  pseudo-metric, invariant measures, unique ergodicity, $\Fbar$-continuity, mean equicontinuity}

\begin{abstract}
 Given a topological dynamical system $(X,T)$, we study properties of the mean orbital  pseudo-metric $\Fbar$ defined by
\[
\Fbar(x,y)=
\limsup_{n\to\infty }
\min_{\sigma\in S_n}\frac{1}{n}\sum_{k=0}^{n-1}d(T^k(x),T^{\sigma(k)}(y)),
\]
where $x,y\in X$ and $S_n$ is the permutation group of $\{0,1,\ldots,n-1\}$.
Let $\V(x)$ denote the  set of measures quasi-generated by a  point $x\in X$.
We show that the map $x\mapsto\V(x)$ is uniformly continuous if $X$ is endowed with the pseudo-metric $\Fbar$ and the space of compact subsets of the set of invariant measures is  considered with the Hausdorff distance. We also obtain a new characterisation of  $\Fbar$-continuity, which connects it to  other properties studied in the literature, like continuous pointwise ergodicity introduced by Downarowicz and Weiss. Finally, we apply our results to reprove some known results on $\Fbar$-continuous and mean equicontinuous systems.
\end{abstract}

\maketitle

\section{Introduction}

A topological dynamical system
is a pair $(X,T)$, where $X$ is a compact metric space  and $T\colon X\to X$ is a continuous map. For the rest of the paper we endow $X$ with a fixed compatible metric $d$. None of the following results depends on the choice of $d$.

We study properties of the \emph{mean orbital pseudo-metric}  $\Fbar$ on $X$ given for $x,y\in X$ by
\begin{equation}\label{eq:mean-metric}
\Fbar(x,y)=
\limsup_{n\to\infty }
\min_{\sigma\in S_n}\frac{1}{n}\sum_{k=0}^{n-1}d(T^k(x),T^{\sigma(k)}(y)),
\end{equation}
where $S_n$ denotes the permutation group on $\{0,1,\ldots,n-1\}$. It is easy to see that $\Fbar$ is a pseudo-metric and
for every $x,y\in X$ and $n,m\ge 0$ we have
$\Fbar(T^nx,T^my)=\Fbar(x,y)$.
%In particular, for every $x\in X$ and $u,v\in\orb(x,T)=\{T^n(x):n\ge 0\}$ we have $\Fbar(u,v)=0$.
Hence, $\Fbar$ depends on the orbits treated as \emph{sets} rather than on individual points or orbits indexed (ordered) by time.

The pseudo-metric $\Fbar$ was recently introduced\footnote{Note that the authors of \cite{ZZ20} did not give $\Fbar$ any name, and denoted it by $\bar{F}$. We decided to change the notation to avoid confusion with another pseudo-metric coined Feldman-Katok pseudo-metric in \cite{KL17} and denoted by $\bar{F}_K$. The change is justified by the fact that $\bar{F}_K$ from \cite{KL17} corresponds to the pseudo-metric $\bar{f}$ introduced for symbolic systems by Feldman, and independently, by Katok in the '70s, while $\bar{F}$ defined by Zheng and Zheng is not equivalent to $\bar{f}$ on symbolic spaces. The latter fact is a consequence of a characterisation of systems with $\bar{F}_K$-diameter $0$ provided in \cite{GRK20} and the fact that systems with $\Fbar$-diameter $0$ are identified as uniquely  ergodic systems in \cite[Theorem 5.1]{ZZ20} (cf. Corollary \ref{cor:uniquely-erogdic}).} by Zheng and Zheng \cite{ZZ20}. It generalises the Besicovitch pseudo-metric $D_B$ given
for every $x,y\in X$ by
\[
D_B(x,y)=
\limsup_{n\to\infty }
\frac{1}{n}\sum_{k=0}^{n-1}d(T^k(x),T^k(y)).
\]
Note that the definition of the mean orbital pseudo-metric $\Fbar$ follows the same scheme as the definition of $D_B$, but defining $\Fbar$ we ignore  the order of points on the orbit.

We continue the study of $\Fbar$ initiated in \cite{ZZ20} and provide new results on $\Fbar$-continuous systems. Recall that $\Fbar$-continuity is another notion introduced in \cite{ZZ20} (under the name \emph{weak mean equicontinuity}). The definition of $\Fbar$-continuity mimics the definition of mean equicontinuous systems. A topological dynamical system $(X,T)$ is \emph{mean equicontinuous} if for every $\eps>0$
there exists a $\delta>0$ such that for every $x,y\in X$
with $d(x,y)<\delta$ one has $D_B(x,y)<\eps$. These systems were originally introduced by Fomin \cite{F51} who called them \emph{stable in the mean in the sense of Lyapunov} or simply \emph{mean-L-stable}. Actually,  Fomin's definition was slightly different, but equivalent to the one using $D_B$ given here. The definition of mean equicontinutiy was originally proposed in \cite{LTY15} and the equivalence was also noted there.  Oxtoby \cite{Ox52}, Auslander \cite{Aus59} and Scarpellini \cite{S82} also
studied  mean equicontinuous systems.  Recently, there was a revived interest in the mean equicontinuous systems: Li et al. \cite{LTY15} (answering a question from  \cite{S82}) and, independently, Garc{\'{\i}}a-Ramos \cite{GR17} proved that any ergodic invariant measure on a mean equicontinuous system has discrete spectrum. We refer to
\cite{LYY19} for further study on mean equicontinuity and related subjects.

Replacing $D_B$ with $\Fbar$ in the definition of a mean equicontinuous system one arrives at the notion of an \emph{$\Fbar$-continuous} topological dynamical system $(X,T)$. Such a system satisfies that for every $\eps>0$ there is a $\delta>0$ such that for every $x,y\in X$ with $d(x,y)<\delta$ we have $\Fbar(x,y)<\eps$.

By extending some results from \cite{ZZ20} (most notably our Corollary \ref{fn} strengthens \cite[Proposition 5.3]{ZZ20}) we are able to provide new proofs for Theorem 1.3 and Proposition 5.1 of \cite{ZZ20}.
 The latter says that $(X,T)$ is uniquely ergodic if and only if the $\Fbar$-diameter of $X$ is $0$, while the former says that $(X,T)$ is $\Fbar$-continuous if and only if ergodic averages of any continuous function converge pointwise to a continuous function on $X$. The latter notion attracted independent interest and was studied recently by Downarowicz and Weiss \cite{DW21} under the name \emph{continuous pointwise ergodicity}.
 %in \cite[Theorem 1.3]{ZZ20}.
%\cite[Theorem 1.3]{ZZ20}, %\cite[Theorem 5.1]{ZZ20}.

Our main result is Theorem \ref{thm-Fbar-continuous} providing an extended characterisation of $\Fbar$-continuous system. We show that $\Fbar$-continuity of $(X,T)$ is equivalent (in particular, see Theorem \ref{thm-Fbar-continuous} below for the full list) to each of the following conditions:
\begin{enumerate}
    \item the conjunction of two conditions: every point $x\in X$ is generic for an ergodic invariant measure $\mu(x)$ and the map $x\mapsto \mu(x)$ is continuous as a map from the metric space $(X,d)$ to the space of ergodic invariant measures $\MTe(X)$ endowed with the weak$^*$ topology (this condition is coined \emph{continuous pointwise ergodicity} of $(X,T)$ in \cite{DW21});
    \item for every continuous function $f\colon X\to\mathbb{R}$ the sequence of ergodic averages
    \[
    \frac{1}{n}\sum_{j=0}^{n-1}f\circ T^j
    \]
    converges uniformly on $X$ (a topological dynamical system satisfying this condition is called \emph{uniform} in \cite{DW21});
    \item the empirical measure maps $\m(\cdot,n)\colon X\to\M(X)$, $n=1,2,\ldots$,  are uniformly equicontinuous on $X$. Here $\M(X)$ denotes the space of all Borel probability measures with the weak$^*$ topology. Furthermore, for $x\in X$ and $n\in\mathbb{N}$ we have
    \[
   \m(x,n)=\frac{1}{n}\sum_{j=0}^{n-1}\hat{\delta}(T^j(x)), %\qquad (n\in\mathbb{N})
    \]
    where $\hat{\delta}(T^j(x))$ stands for the Dirac measure concentrated at $T^j(x)$;
\item the map $\V\colon X\ni x\to\V(x)\in2^{\mathcal{M}(X)}$ is continuous, where $\V(x)$ is the set of the limit points of the sequence $(\m(x,n))_{n=1}^\infty$.
\end{enumerate}
Besides being of independent interest, Theorem \ref{thm-Fbar-continuous} allows us to provide new, shorter proofs of two more results known from the literature: the characterisation of mean equicontinuity via $\Fbar$-continuity of the Cartesian square-product system, which was given in  \cite{FGL18} and the equivalence between mean equicontinuity and Weyl mean equicontinuity proved in \cite{DG16} for minimal systems and in \cite{QZ20} for all topological dynamical systems. The latter  result solves \cite[Question 7.1]{LTY15} and its extension to actions of locally compact $\sigma$-compact amenable groups can be found in \cite{FGL18} (cf. Remark \ref{rem:FGL}). Weyl mean equicontinuity arises by replacing the Besicovitch pseudo-metric $D_B$ in the definition of mean equicontinuous systems with the Weyl pseudo-metric $D_W$. The Weyl pseudo-metric on $(X,T)$ is given for $x,y\in X$ by
\[
D_W(x,y)=\limsup_{n-m\to\infty }
\frac{1}{n}\sum_{k=m}^{n}d(T^k(x),T^k(y)).
\]
The Weyl (also called Banach) mean equicontinuity was introduced in \cite{LTY15}, while the Weyl pseudo-metric in topological dynamics was for the first time considered in \cite{DI88}, and generalised in \cite{LS18}. Our Theorem \ref{thm-Fbar-continuous} applies easily to all isometric $\mathbb{Z}$-actions (see Remark \ref{rem:rotations}).

Although our proofs are presented in the setting of classical topological dynamical systems, they are easily generalised to actions of arbitrary countable Abelian %amenable
groups (see Remark \ref{rem:amenable}).

%The Besicovitch pseudo-metric has attracted a lot of attention, see ....

%The second one
%In this paper, we call is as the permuted mean pseudo-metric.
%That is, for every $x,y\in X$,

%In \cite{KLO17}, the authors studied the relation between
%the Besicovitch pseudo-metric and the space of invariant measures.
%In this paper, we study the relation between
%the permuted mean pseudo-metric and  the space of invariant measures.

\section{Preliminaries}

Let $X$  be a compact metric space with a metric $d$.
Given a point $x\in X$ and a  non-empty subset $A$ of $X$ we define
\[
d(x,A)=\inf\{ d(x,y)\colon y\in A\}.
\]
Let $\M(X)$ be the set of all Borel probability measures on $X$ endowed with the weak$^*$ topology.
We equip $\M(X)$ with the \emph{Prokhorov metric} denoted by $\rho$, that is
for any $\mu,\nu\in \M(X)$, we have
\[\rho(\mu,\nu)=\inf
\{\eps > 0 \colon  \mu(B) \leq  \nu(B^\eps )+\eps\text{  for every Borel set } B\subset X\}.\]
Here (and elsewhere) $B^\eps$ denotes the $\eps$-hull of $B$, that is,
\[
B^\eps =\{x\in X \colon  d(x,B)<\eps\}.
\]

Let $2^X$ be the \emph{hyperspace} of $X$, that is
$2^X=\{A\subset X\colon A$ is a non-empty closed subset of $X\}$.
The \emph{Hausdorff metric} on $2^X$ is defined as follows:
for every $A,B\in 2^X$ we set
\[
d_H(A,B)=\max\bigl\{ \inf\{\eps>0\colon B\subset A^\eps\},
\inf\{\eps>0\colon A\subset B^\eps\}\bigr\}.
\]

Let $T\colon X\to X$ be a continuous, not necessarily invertible map.
The pair $(X,T)$ forms a topological dynamical system. We write $\orb(x,T)$ for the \emph{orbit} of $x\in X$ with respect to $T$, that is, $\orb(x,T)=\{T^n(x):n\ge 0\}$.
 The set of all probability Borel $T$-invariant measures is denoted by $\MT(X)$.
An invariant measure $\mu\in \MT(X)$ is called \emph{ergodic} if for every Borel subset $B$ in $X$, $T^{-1}B=B$ implies that $\mu(B)=0$ or $1$. The set of all ergodic invariant measures is denoted by $\MTe(X)$.
It is well known that $\MT(X)$ is a non-empty, closed, convex subset of $\M(X)$ and $\MTe(X)$ is the set of extreme points of $\MT(X)$.
We say that $(X,T)$ is \emph{uniquely ergodic} if $\MT(X)$ consists of only one measure.
For $x\in X$, let $\hat{\delta}(x)\in \M(X)$ be the Dirac measure supported on $\{x\}$.
Given $x\in X$ and $n\in\mathbb{N}$, we define an \emph{empirical measure}
\[
\m(x,n)=\frac{1}{n}\sum_{i=0}^{n-1}\hat{\delta}(T^ix).
\]

A measure $\mu\in \M(X)$ is a \emph{distribution measure} for
$x\in X$ if $\mu$
is a limit of some subsequence of $\{\m(x,n)\}_{n=1}^\infty$. In this situation, we also say that $x$ \emph{quasi-generates} $\mu$ (synonymously we say that $x$ is a %\emph{quasi-generating} or
\emph{quasi-generic point} for $\mu$). If a measure is a distribution measure, then  it is necessarily invariant.
The set of all distribution measures for $x\in X$ (all measures quasi-generated by $x$) is denoted by $\V(x)$.
We say that $x\in X$ is a \emph{generic} point for $\mu\in M(X,T)$ if $\mu$ is the unique member of $\V(x)$. Every ergodic invariant measure has a generic point, but generic points need not exist for non-ergodic invariant measures. By compactness, for every $x\in X$ the set $\V(x)$ is a non-empty closed connected subset of $\MT(X)$. For a generic point $x\in X$ the unique measure in $\V(x)$ is denoted by $\mu(x)$.
Let $2^{\M(X)}$ be the hyperspace (the space of all non-empty closed subsets of $\M(X)$) and let $\rho_{H}$ be the Hausdorff metric on $2^{\M(X)}$ induced by the Prokhorov metric $\rho$. The assignment $x\mapsto \V(x)$ determines a map from $X$ to $2^{\MT(X)}\subset 2^{\M(X)}$.

\section{The mean orbital pseudo-metric and invariant measures}
Let $(X,T)$ be a topological dynamical system.
Fix $x,y\in X$.
For every $n\in\mathbb{N}$ and $\sigma \in S_n$ we define
\begin{align*}
\Delta_{\sigma}(x, y, \delta)=
\bigl|\{ 0\leq j\leq n-1\colon d(T^j(x), T^{\sigma(j)}(y))>\delta\} \bigr|,
\end{align*}
where $|\,\cdot\,|$ stands for the cardinality of a set,
and
\begin{align*}
   \widetilde{E}(x,y)= \inf\Bigl\{\eps>0 \colon \limsup_{n\to\infty}\frac1n \min_{\sigma\in S_n} \Delta_{\sigma}(x,y, \eps)< \eps\Bigr\}.
\end{align*}
%For further reference, note that for every $\eps>0$, by the definition of $\Bar{F}'$, for every $x$ and  $x'$ in $X$  it holds
%\begin{equation}\label{ineq:12}
%    \Bar{F}'(x, x')< \eps \quad
%    \text{if and only if} \quad \Bar{d}(\Delta_{\sigma}(x, x', \eps))< \eps .
%\end{equation}
%It is easy to see that  $\widetilde{E}$ is a pseudo-metric on $X$, which is uniformly equivalent to $\Fbar$.
The next fact follows from standard techniques (for example, see \cite{Aus59}, \cite{KLO17} or \cite[Lemma 3.1]{LTY15}).
Here we provide a proof for completeness.

\begin{lem}\label{lemma:1}
For a topological dynamical system $(X,T)$,
the pseudo-metrics $\Fbar$ and $\widetilde{E}$
are uniformly equivalent on $X$.
\end{lem}
\begin{proof}
Without loss of generality
we assume that the diameter of $X$ is $1$.
Fix $x,y \in X$.
For every $\delta>0$, $n\in\mathbb{N}$
and $\sigma \in S_n$, we have
\begin{equation}\label{ineq:22}
    \delta \cdot \Delta_{\sigma}(x, y, \delta) \leq \sum_{j=0}^{n-1} d(T^j(x), T^{\sigma{(j)}}(y)) \leq \Delta_{\sigma}(x, y, \delta)+\delta\cdot (n-\Delta_{\sigma}(x, y, \delta)).
\end{equation}
Put $\Delta_n(x, y, \delta)=\min_{\sigma\in S_n} \Delta_{\sigma}(x, y, \delta)$ and
\[
\bar{d}(\Delta( x, y,\delta))=
\limsup_{n\to\infty}\frac1n \Delta_n(x, y, \delta).
\]
It follows from \eqref{ineq:22} that for every $n\in\mathbb{N}$ we have
\begin{equation}\label{ineq:222}
    \delta \cdot \Delta_{n}(x, y, \delta) \leq \min_{\sigma\in S_n} \sum_{j=0}^{n-1} d(T^j(x), T^{\sigma{(j)}}(y)) \leq \Delta_{n}(x, y, \delta)+\delta\cdot (n-\Delta_{n}(x, y, \delta)).
\end{equation}
Dividing \eqref{ineq:222} by $n$ and passing with $n$ to infinity we obtain
\begin{equation}\label{ineq:11}
      \delta \cdot \Bar{d}(\Delta(x, y, \delta)) \leq \Fbar(x, y) \leq (1-\delta)\Bar{d}(\Delta(x, y, \delta)) + \delta .
\end{equation}
Fix $\eps>0$ and
assume that $\delta>0$ satisfies $\delta <\eps^{2}$. If $\Fbar(x,y)<\delta$, then using \eqref{ineq:11} we see that
\begin{equation*}
    \eps \cdot \Bar{d}(\Delta(x, y, \eps)) \leq \Fbar(x, y) < \eps^{2}.
\end{equation*}
Then $\Bar{d}(\Delta(x, y, \eps))<\eps$, and by the definition one has
$\widetilde{E}(x, y)< \eps$.
This implies that the identity map $\text{id} \colon (X, \Fbar)\to (X, \widetilde{E})$ is uniformly continuous.

It remains to show that $\text{id} \colon (X, \widetilde{E})\to (X, \Fbar)$
is uniformly continuous.
Fix $\eps >0$ and take $\delta>0$ such that $\delta <\eps/2$. Assume $\widetilde{E}(x,y)< \delta$,
that is $\Bar{d}(\Delta(x, y, \delta))<\delta$.
We want to show that $\Fbar(x, y)<\eps$.
By $\eqref{ineq:11}$ and $\delta<\eps/2$ we see that $\Fbar(x, y)<\eps$.
\end{proof}

\begin{rem}
By Lemma \ref{lemma:1} it is easy to see that the pseudo-metric $\Fbar$ does not depend on the metric $d$ on $X$,
that is if $d'$ is another compatible metric on $X$, then
the pseudo-metric $\Fbar$ and $\Fbar'$
induced by $d$ and $d'$ are uniformly equivalent. This  result can be also deduced from Lemma \ref{lem:En=Wasser}.
\end{rem}

For $n\in\mathbb{N}$ we define
	\[
	\bar E_n(x,y)=\min_{\sigma\in S_n}
	\frac{1}{n}\sum_{k=0}^{n-1}d(T^kx,T^{\sigma(k)}y).
	\]
	It is easy to see that for each $n\in\mathbb{N}$ the function	$\bar E_n$ is a continuous  metric on $X$.
With this notation, we can rewrite \eqref{eq:mean-metric} as
\begin{equation}\label{eq:mean-metric-En}
   \Fbar(x,y)= \limsup_{n\to\infty}\bar E_n(x,y).
\end{equation}
Now, a crucial observation is that  $\bar E_n(x,y)$ equals the Wasserstein\footnote{The Wasserstein distance is also known as Kantorovich-Rubinstein distance, which probably is a more historically accurate name, see \cite{Vershik}. Note also that Leonid Nisonovich Vaserstein %(Russian: Леонид Нисонович Васерштейн)
after whom the Wasserstein distance is named prefers different spelling of his surname.
We have resisted the temptation to rename things to set the record straight and for the sake of consistency with the existing literature we have kept the traditional name of $W_1$ and its spelling.} distance $W_1$ between the probability measures $\m(x,n)$ and $\m(y,n)$ for $n\in\mathbb{N}$. Recall that a \emph{coupling} of $\mu,\nu\in\M(X)$ is a measure
$\pi\in\M(X\times X)$ such that
$\pi(A\times X)=\mu(A)$ and $\pi(X\times B)=\nu(B)$ for all Borel sets $A,B\subseteq X$. Write $\Pi(\mu,\nu)$ for the convex set of all couplings of $\mu$ and $\nu$. The Wasserstein distance $W_1$ between $\mu,\nu$ %$\mathcal{T}_f(\mu,\nu)$ be
is the cost of optimal transportation (cf. Chapter 21 in \cite{Garling}) between $\mu$ and $\nu$ for the cost function provided by the metric $d$, that is,
\[
%\mathcal{T}_d
W_1(\mu,\nu)=\inf_{\pi\in\Pi(\mu,\nu)}\int_{X\times X} d(x,y)\,\textrm{d}\pi(x,y).
\]
For a compact metric space $X$ the Wasserstein distance $W_1$ is compatible with the weak$^*$ topology on $\M(X)$ (to see this one should combine Thm. 18.6.1, Thm. 18.5.5, Cor. 18.5.6, and Cor. 21.2.4 in \cite{Garling}). Now, the formula
\begin{equation}\label{eq:En=Wasser}
\bar E_n(x,y)= W_1(\m(x,n),\m(y,n))
\end{equation}
can be seen as an easy consequence of the fact that if $\mu$ and $\nu$ are two discrete probability measures and each of them is uniformly distributed over a set of $n$ points in $X$, then the couplings of $\mu$ and $\nu$ are in a one-to-one correspondence with  bistochastic $n\times n$ matrices. The latter fact is certainly known by aficionados (cf. the example given in \cite[pp. 303--304]{Garling}) and then \eqref{eq:En=Wasser} follows from the Birkhoff-von~Neumann theorem (see \cite[Thm. 12.2.11]{Garling}) stating that the extreme points of the convex set of bistochastic $n\times n$ matrices are permutation $n\times n$ matrices. We decided to provide a proof of \eqref{eq:En=Wasser} for completeness.

\begin{lem}\label{lem:En=Wasser}
If $(X,T)$ is a topological dynamical system, then for every $n\in\mathbb{N}$ and $x,y\in X$ the equation \eqref{eq:En=Wasser} holds.
\end{lem}
\begin{proof}
First, note that we can identify a permutation $\sigma\in S_n$ with a discrete measure
\[
\bar\sigma=\sum_{j=0}^{n-1} \hat\delta_{(j,\sigma(j))}
\]
supported on the set $\{0,1,\ldots,n-1\}\times\{0,1,\ldots,n-1\}$. The measure $\frac1n\bar\sigma$ can be also considered as a coupling of
$\m(x,n)$ and $\m(y,n)$ if we interpret  $\hat\delta_{(j,\sigma(j))}$ as a Dirac measure concentrated on the pair $(x_j,T^{\sigma(j)}(y))\in X\times X$.
Furthermore, we have
\begin{multline*}
\Fbar_n(x,y)=\min_{\sigma\in S_n}\frac1n\sum_{j=1}^{n-1} d(x_j,T^{\sigma(j)}(y))=\min_{\sigma\in S_n}\frac1n\int_{X\times X}d(x,y)\,\textrm{d}\bar\sigma(x,y)\\
\geq \inf_{\pi\in\Pi(\m(x,n),\m(y,n))}\int_{X\times X}d(x,y)\,\textrm{d}\pi(x,y) =  W_1(\m(x,n),\m(y,n)).
\end{multline*}
For the reverse inequality, we note that every coupling $\pi\in\Pi(\m(x,n),\m(y,n))$ corresponds to a bistochastic matrix $\Xi$ whose $(i+1,j+1)$ entry is given by $n\pi(\{(T^i(x),T^j(y)\})$, where $0\le i,j\leq n-1$. By the Birkhoff-von Neumann theorem (see \cite[Thm. 12.2.11]{Garling}) $\Xi$ must be a convex combination of permutation matrices, which in turn correspond to measures $\bar\sigma$ with $\sigma\in S_n$. Therefore we can write
\[
\pi=\sum_{\sigma\in S_n}\lambda_\sigma\frac1n\bar\sigma,
\]
where $\lambda_\sigma$ ($\sigma\in S_n$) are nonnegative real coefficients summing up to $1$. It follows that
\[
\int_{X\times X}d(x,y)\,\textrm{d}\pi(x,y)=\sum_{\sigma\in S_n}\lambda_\sigma\frac1n\int_{X\times X}d(x,y)\,\textrm{d}\bar\sigma(x,y).
\]
We conclude that for every $\pi\in\Pi(\m(x,n),\m(y,n))$ it holds
\[
\int_{X\times X}d(x,y)\,\textrm{d}\pi(x,y)
\ge \min_{\sigma\in S_n}\frac1n\int_{X\times X}d(x,y)\,\textrm{d}\bar\sigma(x,y)=\Fbar_n(x,y).
\]
Thus, $W_1(\m(x,n),\m(y,n))\ge \Fbar_n(x,y)$, which finishes the proof.
\end{proof}

\begin{cor}\label{fn}\label{cor:En-Wasser}
    Let $(X,T)$ be a topological dynamical system.
    \begin{enumerate}
        \item For every $\eps>0$, there exists $\delta > 0$ such that for every $x,y\in X$ and $n\in \mathbb{N}$ with $\rho(\m(x,n),\m(y,n))<\delta$ we have $\Fbar_n(x,y)<\eps$.
        \item For every $\eps>0$, there exists $\delta > 0$ such that for every $x,y\in X$  and $n\in \mathbb{N}$  with $\Fbar_n(x,y)<\delta$ we have   $\rho(\m(x,n),\m(y,n))<\eps$.
    \end{enumerate}
\end{cor}
\begin{proof}Since $W_1$ and $\rho$ are compatible with the weak$^*$ topology on $\M(X)$, the
corollary follows immediately from the fact that on a compact metrizable space all compatible metrics are uniformly equivalent.
\end{proof}
\begin{lem}\label{lemma:2}
Let $(X,T)$ be a topological dynamical system.
For every $\eps>0$ there is a $\delta>0$ such that for every $x,y\in X$ with  $\Fbar(x, y)<\delta$ there exists an $N=N(x,y)\in\mathbb{N}$ so that $\rho(\m(x, n), \m(y,n))<\eps$ for all $n\ge N$.
\end{lem}
\begin{proof}
Given $\eps>0$ we use Corollary \ref{cor:En-Wasser} to find  $\delta>0$ such that for every $n\in\mathbb{N}$ and  $x, y\in X$ we have
\begin{equation}\label{eq:imp}
    \Fbar_n(x,y)<\delta\implies\rho(\m(x,n),\m(y,n))<\eps.
\end{equation}
Now, if $\Fbar(x,y)<\delta$, then $\Fbar_n(x,y)<\delta$ for all sufficiently large $n$'s (cf. \eqref{eq:mean-metric-En}), so $\rho(\m(x, n), \m(y, n))<\eps$ for all $n\ge N=N(x,y)$ by \eqref{eq:imp}.
\end{proof}

The version of the following result for Besicovitch pseudo-metric was first proved in \cite{KLO17}. Since we always have $\Fbar(x,y)\le D_B(x,y)$, our Theorem \ref{theorem:1} contains \cite[Theorem 7]{KLO17} as a special case.

\begin{thm}\label{theorem:1}
For a topological dynamical system $(X,T)$,
the map $(X, \Fbar)\to (2^{\MT(X)}, \rho_H)$,
$x\mapsto \V(x)$ is uniformly continuous,
that is for every $\eps>0$ there exists a $\delta>0$ such that
for every $x,y\in X$ with $\Fbar(x,y)<\delta$ one has
$\rho_H(\V(x),\V(y))<\eps$.
\end{thm}
\begin{proof}
Fix $\eps>0$.  Applying Lemma \ref{lemma:2} for $\eps/6$ in place of $\eps$, we get a desired constant $\delta>0$. Without loss of generality, assume that $\delta<\eps/6$. Choose $x, y \in X$ such that $\Fbar(x,y) < \delta$.
Pick $\mu\in \V(x)$.
Find a subsequence $\{n_k\}_{k=0}^{\infty}$ such that $\m(x, n_k)\to \mu$ as $k\to\infty$. Let $N_1\in\mathbb{N}$ be such that for all $k$ with $n_k\ge N_1$ it holds that $\rho(\mu, \m(x, n_k))<\eps/6$.
Without loss of generality, assume that $\m(y, n_k)\to \nu$  as $k\to\infty$. Then $\nu \in\V(y)$ and we can find $N_2\in\mathbb{N}$ be such that for all $k$ with $n_k\ge N_2$ it holds that $\rho(\nu, \m(y, n_k))<\eps/6$.
By Lemma \ref{lemma:2}, there exists $N=N(x,y)\in\mathbb{N}$ such that for all $k$ with $n_k\ge N$, we have $\rho(\m(x, n_k), \m(y, n_k))\leq \eps/6$. Now for $k$ with $n_k\ge \max\{N,N_1,N_2\}$ we have
\[\rho(\mu, \nu)\leq \rho(\mu, \m(x, n_k))+\rho(\m(x, n_k), \m(y, n_k))+ \rho (\m(y, n_k), \nu)\le \eps/2.
\]
Therefore,
\begin{equation*}
    \rho(\mu, \V(y))=\min \{ \rho(\mu, \eta)\colon  \eta \in \V(y)\}\le \eps/2.
\end{equation*}
Hence $\V(x)\subset (\V(y))^\eps$ and
exchanging the roles of $x$ and $y$ in the argument above we have $\V(y)\subset (\V(x))^\eps$. Thus $\rho_H(\V(x), \V(y))< \eps$.
%Thus $\rho_H(m(x, n_k), m(y, n_k))< \eps$.
\end{proof}

\begin{cor} \label{cor:fbar-0}
Let $(X,T)$ be a topological dynamical system.
For every $x,y\in X$  if $\Fbar(x,y)=0$ then
$\V(x)=\V(y)$.
\end{cor}

It is natural to ask that whether the converse of Corollary \ref{cor:fbar-0} is true.
We will show in subsection \ref{subsec:example} that in general, the converse need not to hold. The example shows that $\V(x)=\V(y)$ without $\m(x,n)$ and $\m(y,n)$ being close to each other for all sufficiently large $n$. The example also shows that the assumption that $x$ is a generic point in Corollary \ref{thm1} is indispensable.
%Nevertheless, some kind of converse is possible, when we assume that $x$ is a generic point. %we have the following result.

\begin{cor}\label{thm1}
Let $(X,T)$ be a topological dynamical system. If $x\in X$ is a generic point, then for every $\eps>0$, there is a $\delta > 0$ such that for every $y\in X$
with $\rho_H(\V(x),\V(y))<\delta$ one has
$ \Fbar(x,y)\le \eps$.	
\end{cor}
\begin{proof}
Assume $\V(x)=\{\mu\}$.
Let $\eps>0$. By Corollary \ref{fn},
	there exists a $\delta > 0$  such that for every $y\in X$ and $n\in \mathbb{N}$,
we have $$\rho(\m(x,n),\m(y,n))<\delta\Rightarrow\Fbar_n(x,y)<\eps.$$
Fix a point $y\in X$ with	$\rho_H(\V(x),\V(y))<\delta$.
Note that there exists a strictly increasing sequence $\{n_k\}_{k=1}^\infty$ such that
\[
\Fbar_{n_k}(x,y)\to \Fbar(x,y)\quad (k\to\infty).
\]
Passing to a subsequence once more we can assume
$\m(y,n_k)\to \nu$ with $k\to\infty$.
Then $\nu\in \V(y)$ and $\rho(\mu,\nu)<\delta$.
Note that
$\m(x,n_k)\to \mu$ as $k\to\infty$. There exists an $N_0\in\mathbb{N}$ such that
$\rho(\m(x,n_k),\m(y,n_k))<\delta$ for all $k>N_0.$
Hence for all $k>N_0$, we have $\Fbar_{n_k}(x,y)<\eps.$
Hence
\[\lim_{k\to \infty}\Fbar_{n_k}(x,y)=\Fbar (x,y)\leq \eps.\qedhere\]
\end{proof}

As an immediate consequence we obtain new proof of the following two facts, originally proved in \cite{ZZ20}.

\begin{cor}[{\cite[Proposition 5.3]{ZZ20}}] \label{cor:fbar0}
Let $(X,T)$ be a topological dynamical system. If $x\in X$ is a generic point, then for every $y\in X$,
$\V(x)=\V(y)$ if and only if $\Fbar(x,y)=0$.
\end{cor}

\begin{cor}[{\cite[Theorem 5.1]{ZZ20}}] \label{cor:uniquely-erogdic}
A topological dynamical system  $(X,T)$ is uniquely ergodic
if and only if  $\Fbar(x,y)=0$ for all $x,y \in X$.
\end{cor}

\begin{rem}
Corollary \ref{cor:uniquely-erogdic} and the results of \cite{GR17,LTY15}, and \cite{GRK20} show that (minimal) mean equicontinuous systems form a proper subclass of (minimal) Feldman-Katok equicontinuous systems (see \cite{GRK20} for more details), and the latter class is properly contained in the class of (minimal) $\Fbar$-continuous topological dynamical systems.
\end{rem}

\begin{cor}\label{cor:erg}
Let $(X,T)$ be a topological dynamical system. If $\mu,\nu\in\MT(X)$ then for every $x\in X$ which is generic for $\mu$ and $y\in X$ generic for $\nu$ we have
\[
\Fbar(x,y)=\lim_{n\to\infty} \Fbar_n(x,y)=W_1(\mu,\nu).
\]
Furthermore, if $\mu_n\in\MT(X)$ and $x_n\in X$ is its generic point for $n=1,2,\ldots$, then for every $x\in X$ we have $\Fbar(x_n,x)\to 0$ as $n\to\infty$ if and only if $x$ is a generic point for a measure $\mu\in\MT(X)$ such that $W_1(\mu_n,\mu)\to 0$ as $n\to\infty$.
\end{cor}
\begin{proof}
Note that the ``furthermore'' statement follows from the first part. For the proof of the first part, simply note that
\[
W_1(\mu,\nu)=\lim_{n\to\infty}W_1(\m(x,n),\m(y,n))=\lim_{n\to\infty}\Fbar_n(x,y).
\qedhere\]
\end{proof}
Note that a weak$^*$ limit of a sequence of ergodic measures need not to be ergodic, therefore even if we assume that $\mu_n\in\MTe(X)$ for every $n\in\mathbb{N}$ in the Corollary \ref{cor:erg}, we still cannot conclude that the measure $\mu$ is ergodic. Roughly speaking, the $\Fbar$-limit of ergodic points need not be an ergodic point. Note that results saying that ergodicity is a closed property are available for the Besicovitch pseudometric \cite[Thm. 15]{KLO17} and Feldman-Katok pseudometric \cite{KL17}.  Corollary \ref{cor:erg} shows that these results cannot be extended to $\Fbar$-pseudometric.
\subsection{An Example} \label{subsec:example}
In this subsection we will provide an example showing that $\V(x)=\V(y)$ need not imply $\Fbar(x,y)=0$. This shows that the assumption that $x$ is a generic point can not be omitted %in Theorem~\ref{fn}, Corollaries~\ref{thm1} and~\ref{cor:fbar0}.
neither in Corollary \ref{thm1} nor in  Corollary \ref{cor:fbar0}.

%Let $X$ be the full (one-sided) shift over the alphabet $\{0,1\}$.
We endow the space of all infinite sequences of symbols from
$\{0,1\}$ indexed by the non-negative integers with the product topology, and denote it by $\{0,1\}^\infty$.
A compatible metric  $\{0,1\}^\infty$ is given by
for any $x=(x_n)_{n=0}^\infty, y=(y_n)_{n=0}^\infty \in \{0,1\}^\infty$
\begin{equation*}
d(x,y)=\begin{cases}
0, & \text{ if } x=y; \\
2^{-k}, & k= \min\{i\geq 0\colon x_i\neq y_i\}.
\end{cases}
\end{equation*}
The shift map $S$ on $\{0,1\}^\infty$ is defined by $S(x_0x_1x_2\ldots)=x_1x_2x_3\ldots$.
Note that given two points $x=(x_n)_{n=0}^\infty$ and $y=(y_n)_{n=0}^\infty$ in the full shift, some $N\ge 1$ and a permutation $\sigma\in S_N$ the sum
\[
\sum_{k=0}^{N-1}d(S^k(x),S^{\sigma(k)}(y))=
\sum_{k=0}^{N-1}d(x_kx_{k+1}\ldots,y_{\sigma(k)}y_{\sigma(k)+1}\ldots)
\]
is bounded below by the number of $0\le k <N$ such that $x_k\neq y_{\sigma(k)}$. This is because $d(x_kx_{k+1}\ldots,y_{\sigma(k)}y_{\sigma(k)+1}\ldots)=1$ provided $x_k\neq y_{\sigma(k)}$. Therefore,
%if the symbol $1$ appears in $x_{[0,N)}$ and
%$y_{[0,N)}$ different number of times, then $\Fbar (x,y)$
\begin{equation}\label{ineq:fbar-symb-lower-bound}
\min_{\sigma\in S_N}\sum_{k=0}^{N-1}d(S^k(x),S^{\sigma(k)}(y))\ge \left\vert
\sum_{k=0}^{N-1}x_k-\sum_{k=0}^{N-1}y_k
\right\vert.
\end{equation}

Let $(a_n)_{n\ge 1}$ be a sequence of positive integers such that
\begin{equation}\label{eq:dominance}
\frac{a_1+\dotsb+a_n}{a_{n+1}}\to 0\quad\text{as }n\to\infty.
\end{equation}
We also define an auxiliary sequence
$(b_n)_{n\ge 1}$ by
\[
b_n=a_1+\dotsb+a_n \quad\text{for }n=1,2,3\ldots
\]
Define two sequences of finite words over $\{ 0, 1\}$, denoted
$(U_n)_{n\ge 1}$ and $(V_n)_{n\ge 1}$ by
\[
U_n = \begin{cases}0^{a_n} & \text{if $n$ is odd,} \\ 1^{a_n} & \text{if $n$ is even,}
\end{cases}
\quad V_n = \begin{cases}1^{a_n} & \text{if $n$ is odd,} \\ 0^{a_n} & \text{if $n$ is even.}
\end{cases}
\]
Define
\[
x=(x_n)_{n=0}^\infty=U_1U_2U_3\ldots\quad\text{and}\quad y=(y_n)_{n=0}^\infty=V_1V_2V_3\ldots
\]
We claim that $\V(x)=\V(y)$, but $\Fbar(x,y)=1$. We first prove the latter equality.
It follows from \eqref{ineq:fbar-symb-lower-bound} that for every $n\ge 1$ we have
\begin{equation}\label{ineq:ones}
\frac{1}{b_n}
\min_{\sigma\in S_{b_n}}\sum_{k=0}^{b_n-1}d(S^k(x),S^{\sigma(k)}(y))\ge \frac{1}{b_n}(a_n - (a_1+\ldots+a_{n-1})).
\end{equation}
As a consequence of \eqref{eq:dominance} we see that the right-hand side of \eqref{ineq:ones} tends to $1$ as $n\to\infty$. This proves that $\Fbar(x,y)=1$.

Let
\[
K=\{\alpha  \delta_{\bar{0}} + (1-\alpha) \delta_{\bar{1}} : \alpha \in [0, 1]\},\]
where $\delta_{\bar{0}}$, respectively $\delta_{\bar{1}}$ is the Dirac measure concentrated on the fixed point $\bar{0}=000\ldots$, respectively $\bar{1}=111\ldots$.
%We will now prove that
It is easy to see that
$\V(x)=\V(y)=K$.

\section{\texorpdfstring{$\Fbar$-continuity and mean equicontinuity}{F-bar-continuity and mean equicontinuity}}

In this section, we provide an extended characterisation of $\Fbar$-continuity (Theorem \ref{thm-Fbar-continuous}). The $\Fbar$-continuity was first introduced in \cite{ZZ20}.
Recall that a topological dynamical system $(X,T)$ is called \emph{$\Fbar$-continuous} if for every $\eps>0$
there exists a $\delta>0$ such that for every $x,y\in X$
with $d(x,y)<\delta$ one has $\Fbar(x,y)<\eps$.
We first note a corollary to Lemma \ref{lemma:2}, which is crucial for our approach in this section.

\begin{lem}\label{lem:unique-erg}
If  $(X,T)$ is a topological dynamical system such that the map $\V\colon (X,d)\to (2^{\MT(X)},\rho_H)$, $x\mapsto \V(x)$ is continuous, then for every $x\in X$ the topological dynamical system $(\overline{\orb(x,T)},T)$ is uniquely ergodic.
\end{lem}
\begin{proof}
 It is clear that the map $\V$ is constant on $\orb(x,T)$. By the continuity, $\V$ must be also constant on  $\overline{\orb(x,T)}$.
But there exists an ergodic measure $\nu$ supported on $\overline{\orb(X,T)}$ and $\nu$ has a generic point $y\in\overline{\orb(X,T)}$, so $\V(z) = \V(y) =\{\nu\}$ for every $z\in\overline{\orb(X,T)}$.
\end{proof}

The following is essentially Corollary 5.2 of \cite{ZZ20}.
Here we give a proof for completeness.
\begin{lem}\label{lem-Fbar-continuous}
Let $(X,T)$ be a topological dynamical system.
If $(X,T)$ is $\Fbar$-continuous and has a dense orbit
then $(X,T)$ is uniquely ergodic.
\end{lem}
\begin{proof}
Pick a point $x\in X$ such that $\orb(x,T)$ is dense in $X$.
Then $\Fbar(a,b)=0$ for every $a,b\in\orb(x,T)$.
As $(X,T)$ is $\Fbar$-continuous and $X=\overline{\orb(x,T)}$,
we have $\Fbar(a,b)=0$ for all $a,b\in X$.
Then $(X,T)$ is uniquely ergodic by Corollary \ref{cor:uniquely-erogdic}.
\end{proof}
\begin{rem}\label{rem:H-cont}
Note that the Hausdorff metric $\rho_H$ agrees with $\rho$ on the space of all singletons in $\M(X)$, that is, we have
\[
\rho_H(\{\mu\},\{\nu\})=\rho(\mu,\nu) \quad\text{for every }\mu,\nu\in \M(X).
\]
Therefore if $(X,T)$ is a topological dynamical system such that every $x\in X$ is generic for some measure, then the map $x\mapsto \mu(x)$ is well defined and its continuity as a map from $(X,d)$ to $(\MT(X),\rho)$ is equivalent to continuity of the map $x\mapsto \V(x)$ treated as a map from $(X,d)$ to $(2^{\MT(X)},\rho_H)$.
\end{rem}
We have the following characterisation of $\Fbar$-continuity.
\begin{thm}\label{thm-Fbar-continuous}
Let $(X,T)$ be a topological dynamical system. Then the following statements are equivalent:
\begin{enumerate}[label=(\arabic*),ref=(\arabic*)]
    \item \label{Fbar-equi-1} $(X,T)$ is $\Fbar$-continuous;
    \item\label{Fbar-equi-1B} the map $\V\colon (X,d)\to (2^{\MT(X)},\rho_H)$ is continuous;
    \item\label{Fbar-equi-1A}
    the family of empirical measure maps
    $(\m(\cdot,n))_{n=1}^\infty$ is uniformly equicontinuous on $X$, where $\m(\cdot,n)\colon X\to \M(X)$ is given by
    \[
    x\mapsto \m(x,n)=\frac{1}{n}\sum_{j=0}^{n-1}\hat{\delta}(T^j(x))\qquad \text{for $x\in X$ and $n\in\mathbb{N}$};
    \]
    %is uniformly equicontinuous on $X$;
    %the
    %\[
    %x\mapsto \m(x,n)=\frac{1}{n}\sum_{j=0}^{n-1}\hat{\delta}(T^j(x))\qquad (n\in\mathbb{N}),
    %\]
    %are uniformly equicontinuous on $X$;
    \item \label{Fbar-equi-2} for every $x\in X$ the topological dynamical system $(\overline{\orb(x,T)},T)$ is uniquely ergodic and the map $X\ni x\mapsto \mu(x)\in\MTe(X)$ is continuous;
    \item \label{Fbar-equi-3} every $x\in X$ is generic for some ergodic invariant measure  and the map $X\ni x\mapsto \mu(x)\in\MTe(X)$ is continuous;
    \item \label{Fbar-equi-4} every $x\in X$ is generic for some invariant measure  and the map $X\ni x\mapsto \mu(x)\in\MT(X)$ is continuous;
    \item \label{Fbar-equi-5} for every continuous function $f\colon X\to \mathbb{R}$,
    the sequence of continuous functions  $\{\frac{1}{n}\sum_{k=0}^{n-1}f\circ T^k\}$ is pointwise convergent to a continuous function $f^*$;
    \item  \label{Fbar-equi-6} for every continuous function $f\colon X\to \mathbb{R}$,
    there exists a continuous function $f^*\colon X\to \mathbb{R}$ such that for every $x\in X$
    we have
    \[
    \lim_{n-m\to\infty} \frac{1}{n-m}\sum_{k=m}^{n-1}f(T^k(x))=f^*(x);
    \]
    \item \label{Fbar-equi-7} for every continuous function $f\colon X\to \mathbb{R}$,
    the sequence of continuous functions  $\{\frac{1}{n}\sum_{k=0}^{n-1}f\circ T^k\}$ converges uniformly to a continuous function $f^*$.
    %to a continuous function $f^*$;
\end{enumerate}
If any of the conditions \ref{Fbar-equi-1}--\ref{Fbar-equi-7} holds, then the limit continuous function $f^*$ mentioned in \ref{Fbar-equi-5}--\ref{Fbar-equi-7} is the function  given for $x\in X$ by
\[
f^*(x)=\int_X f \dd\mu(x).
\]
\end{thm}
\begin{proof}
\Implies{Fbar-equi-1}{Fbar-equi-1B}: By definition of $\Fbar$-continuity the identity map mapping $(X,d)$ to $(X,\Fbar)$ is continuous. With this observation, the conclusion follows from Theorem  \ref{theorem:1}.

\Implies{Fbar-equi-1B}{Fbar-equi-1A}: By Lemma \ref{lem:unique-erg} every $x\in X$ is a generic point so the map $x\mapsto\mu(x)$ is well-defined and continuous (see Remark \ref{rem:H-cont}).

First, we claim that for every $x\in X$ and $\eps>0$ there exist $N=N(x)\in\mathbb{N}$ and $\delta=\delta(x)>0$ such that if $y\in X$ and $d(x,y)<\delta$, then $\rho(\m(x,n),\m(y,n))<\eps$ for every $n\ge N$.

Before proving the claim, we will show that it is all we need to finish the proof. To this end, we fix $\eps>0$ and use the claim for every $x\in X$ to find $\delta(x)>0$ and $N(x)\in\mathbb{N}$, so that for every $y\in X$ with $d(x,y)<\delta(x)$  we have
\begin{equation} \label{ngeNx}
\rho(\m(x,n),\m(y,n))<\eps/2\quad\text{for all $n\ge N(x)$}.
\end{equation}
Next, we cover the space $X$ by open balls $B(x,\delta(x))$ and choose a finite subcover $\mathcal{U}$. Denote the centres of the balls in $\mathcal{U}$ by $x_1,\ldots,x_k$. Let $N=\max\{N(x_1),\ldots,N(x_k)\}$. Let $\delta_0>0$ be the Lebesgue number for $\mathcal{U}$. We use continuity of the maps $\m(\cdot,1),\ldots,\m(\cdot,N)$ to find $\delta_1>0$ so that if $x,y\in X$ satisfy $d(x,y)<\delta_1$, then
\begin{equation}\label{1N-bound}
\rho(\m(x,n),\m(y,n))<\eps\quad\text{for }1\le n\le N.
\end{equation}
Let $\delta=\min\{\delta_0,\delta_1\}$. To see that our choice of $\delta$ is right, we pick $x,y\in X$ with $d(x,y)<\delta$. Note that this choice of $\delta$ implies that \eqref{1N-bound} holds. Since $\{x,y\}$ has the diameter smaller than the Lebesgue number for the cover $\mathcal{U}$ there is $1\le j \le k$ such that $d(x,x_j)<\delta(x_j)$ and $d(y,x_j)<\delta(x_j)$. Using that $N(x_j)\le N$ and the triangle inequality it now easily follows from \eqref{ngeNx} that
\begin{equation}\label{Nlen-bound}
\rho(\m(x,n),\m(y,n))<\eps\quad\text{for }N\le n.
\end{equation}
Combining \eqref{1N-bound} and \eqref{Nlen-bound} we see that %Lemma \ref{lem:rho-equicont}\eqref{unif-equi-2}
the condition \ref{Fbar-equi-1A} of Theorem \ref{thm-Fbar-continuous} holds as required.

It remains to prove the claim holds. Aiming for a contradiction, suppose that there are $x\in X$ and $\eps_0>0$ such that for every $k,N\in\mathbb{N}$ we can find $y_k\in X$ and $n_k\ge N$  satisfying $d(x,y_k)<1/k$ and
\begin{equation}\label{eps0}
    \rho(\m(x,n_k),\m(y_k,n_k))\ge\eps_0.
\end{equation}
Since given $k\in\mathbb{N}$ our choice of $N\in\mathbb{N}$ is arbitrary, we can use the fact that $x$ is generic for $\mu(x)$ and pick $N=N(k)$ such that
\begin{equation}\label{distmxnkmux}
    \rho(\m(x,n),\mu(x))\le 1/k\quad\text{for all }n\ge N=N(k).
\end{equation}
In particular, \eqref{distmxnkmux} holds with $n_k$ substituted for $n$.
Note that the way we have chosen $y_k$ implies that $y_k\to x$ with $k\to\infty$. In addition,
passing to a subsequence, if necessary, we assume that the sequence
$(\m(y_k,n_k))_{k=1}^{\infty}$ converges to $\xi$ when $k\to\infty$. We necessarily have that $\xi$ is an invariant measure, but it is not necessarily true that $\xi$ is ergodic. By the ergodic decomposition theorem, there is a probability measure $\lambda$ supported on the set $\MTe(X)$ of ergodic $T$-invariant measures such that
\begin{equation}
    \label{ergd}
\xi=\int_{\MTe(X)}\nu \dd\lambda(\nu).
\end{equation}
%In particular, for every continuous function $f$ on $X$ we have
%\[
%\int_X f\,d\mu=\int_{M^e_T(X)}(\int_X f\,d\nu)\,d\lambda(\nu).
%\]
Let $\text{supp}(\xi)$ be the support of $\xi$, that is, the smallest closed set with the full  measure with respect to $\xi$. By \eqref{ergd} we get that $\nu(\text{supp}(\xi))=1$ for $\lambda$-almost every $\nu\in \MTe(X)$. By \eqref{eps0}, we also have $\rho(\nu,\mu(x))>0$ for $\nu$ in a subset of $\MTe(X)$ with positive measure with respect to $\lambda$.

We conclude that there must be an ergodic measure $\nu$ such that $\text{supp}(\nu)\subset\text{supp}(\xi)$ and $\rho(\nu,\mu(x))>0$.
Since points generic for $\nu$ lie densely in
$\text{supp}(\nu)$ and $z\mapsto \mu(z)$ is assumed to be continuous, any point $z\in \text{supp}(\nu)$ satisfies $\mu(z)=\nu$ and $\xi(U)>0$ for every open  neighbourhood $U$ of $z$.  By the portmanteau theorem, for each open neighbourhood $U$ of $z$ we have
\[
\liminf_{k\to\infty}\m(y_k,m_k)(U)=\frac{1}{n_k}|\{0\le j < n_k: T^j(y_k)\in U\}|\ge \xi(U)>0.
\]
It follows that for some sequence $(j_k)_{k=1}^\infty$ we have $T^{j_k}(y_k)\to z$ when $k\to\infty$.
Continuity of the map $\mu(\cdot)$ implies that
\[
\lim_{k\to\infty}\mu(T^{j_k}(y_k))=\mu(\lim_{k\to\infty}T^{j_k}(y_k))=\mu(z)=\nu.
\]
Now, since for each $k\in\mathbb{N}$ we have $\mu(T^{j_k}(y_k))=\mu(y_k)$, we conclude that $\mu(x)=\mu(z)=\nu$ contradicting  $\rho(\nu,\mu(x))>0$.

\Implies{Fbar-equi-1A}{Fbar-equi-2}: Since $\V(x)$ is the set of limit points of the sequence $(\m(x,n))_{n=1}^\infty$ the continuity of the map $x\mapsto \V(x)$ follows immediately from \ref{Fbar-equi-1A}. To finish the proof it is enough to apply Lemma \ref{lem:unique-erg}.

%\Implies{Fbar-equi-1}{Fbar-equi-2}:
%[(1)$\Rightarrow$(2)]
%Fix $x\in X$. Then $(\overline{\orb(x,T)},T)$ is also $\Fbar$-continuous. By Lemma \ref{lem-Fbar-continuous},
%$(\overline{\orb(x,T)},T)$ is uniquely ergodic, hence $\V(x)$
%is a singleton. By Remark \ref{rem:H-cont}, continuity of the map $x\mapsto \mu(x)$ follows from the definition of $\Fbar$-continuity and Theorem \ref{theorem:1}.

\Implies{Fbar-equi-2}{Fbar-equi-3}: By Remark \ref{rem:H-cont}, continuity of the map $x\mapsto \mu(x)$ follows from \ref{Fbar-equi-2}. The rest of the proof is clear.

\Implies{Fbar-equi-3}{Fbar-equi-4}: This is clear.

%[(4)$\Rightarrow$(1)]
\Implies{Fbar-equi-4}{Fbar-equi-1}:
This follows from %the definition of $\Fbar$-continuity and
Corollary \ref{thm1}.

%\Implies{Fbar-equi-2}{Fbar-equi-6}:
%[(2)$\Rightarrow$(6)]
%Since  $(\overline{\orb(x,T)},T)$ is uniquely ergodic for every $x\in X$ we have
%\[
%    \lim_{n-m\to\infty} \frac{1}{n-m}\sum_{k=m}^{n-1}f(T^kx)=\int_X f d\mu(x).
%    \]
%In particular, the limit above exists.
%Since the integral of a continuous function depends continuously on the measure in the weak$^*$ topology, the continuity of function $f^*(x)=\int_X f d\mu(x)$ follows from continuity of $x\mapsto\mu(x)$.

\Implies{Fbar-equi-1A}{Fbar-equi-7}:
%[(2)$\Rightarrow$(6)]
Note that
for every $x\in X$ and $n\in\mathbb{N}$, the following holds
\[\frac{1}{n}\sum_{k=0}^{n-1}f\circ T^k(x)=\int_X f \dd\m(x,n).\]
Therefore uniform equicontinuity of the family of maps $(m(\cdot,n))_{n=1}^\infty$%, where $m(\cdot,n)\colon X\to \M(X)$
 implies that for every continuous function $f\colon X\to \mathbb{R}$,
    the family of continuous functions  \[\left\{\frac{1}{n}\sum_{k=0}^{n-1}f\circ T^k: n\in\mathbb{N}\right\}\] is also uniformly equicontinuous. To end the proof it suffices to follow the proof of (2)$\Rightarrow$(3) in \cite[Lemma 3.3]{LTY15} almost verbatim (the same reasoning is implicit in \cite{Ox52}).

\Implies{Fbar-equi-7}{Fbar-equi-6}: Fix a continuous function $f\colon X\to \mathbb{R}$.
    Assume that the sequence $\{\frac{1}{n}\sum_{k=0}^{n-1}f\circ T^k\}$ converges uniformly to $f^*\colon X\to \mathbb{R}$. Note that $f^*=f^*\circ T^m$ for every $m\in\mathbb{N}$. Fix $\eps>0$ and use uniform convergence to find $L\in\mathbb{N}$ such that for every $\ell\ge L$ and $x\in X$ we have
    \[
    \left|\frac{1}{\ell}\sum_{k=0}^{\ell-1}f\circ T^k(x)-f^*(x)\right|<\eps.
    \]
It follows that if $n-m\ge L$ and $x\in X$, then
%\begin{multline*}
\[    \left|\frac{1}{n-m}\sum_{k=m}^{n-1}f\circ T^k(x)-f^*(x)\right|
    =\left|\frac{1}{n-m}\sum_{k=0}^{n-m-1}f\circ T^k(T^m(x))-f^*(T^m(x))\right|<\eps
\]
%\end{multline*}
and \ref{Fbar-equi-6} is proved.

\Implies{Fbar-equi-6}{Fbar-equi-5}: This is clear.

%\Implies{Fbar-equi-2}{Fbar-equi-6}:
%Fix a point $x\in X$. As $(\overline{\orb(x,T)},T)$ is uniquely ergodic,
%one has \[
%\lim_{n-m\to\infty}\frac{1}{n-m}\sum_{i=0}^{n-1}\hat{\delta}(T^ix)=\mu(x).
%\]
%Then for any continuous function $f\colon X\to \mathbb{R}$, one has
% \[
%    \lim_{n-m\to\infty} \frac{1}{n-m}\sum_{k=m}^{n-1}f(T^kx)=\int f\, d\mu(x).
%\]
%The map $x\mapsto \mu(x)$ is continuous, then so is the map $x\mapsto \int f\, d\mu(x)$.

\Implies{Fbar-equi-5}{Fbar-equi-4}: It is a consequence of the fact that $\mu_n\to \mu$ as $n\to\infty$ in $\M(X)$ if and only if
$\int f \dd\mu_n\to \int f \dd\mu$ as $n\to\infty$ for every continuous function $f\colon X\to\mathbb{R}$.
\end{proof}

\begin{rem}
Note that the equivalence  \IFF{Fbar-equi-1}{Fbar-equi-5} in Theorem \ref{thm-Fbar-continuous} was also proved in \cite[Theorem 1.3]{ZZ20}.
The proof here is new.
\end{rem}
\begin{rem}
One may wonder if the condition saying that the map $x\mapsto\mu(x)$ is continuous may be omitted from Theorem \ref{thm-Fbar-continuous}\ref{Fbar-equi-3}. Dowker and Lederer \cite{DL}  seems to be the first to investigate the question of what happens when all points of a topological dynamical system are generic for some invariant measure. It turns out
that this property together with minimality implies unique ergodicity \cite{DL}.
Katznelson and Weiss  \cite{KW} proved that
if all points are generic and there is only one minimal set, but  more than one ergodic measure, then there must be uncountably
many ergodic measures, and showed that this
possibility can actually arise. In the Katznelson--Weiss example there exists exactly one fully supported ergodic invariant measure and every point is generic for some ergodic invariant measure, which means that the map $x\mapsto\mu(x)$ cannot be continuous. It follows that weakening of Theorem \ref{thm-Fbar-continuous}\ref{Fbar-equi-3} is not possible. Another example of this type is constructed in \cite{FKKL}.
\end{rem}
\begin{rem}\label{rem:rotations}
Note that Theorem \ref{thm-Fbar-continuous} provides an almost purely topological proof that a group rotation is uniquely ergodic if and only if it has a dense orbit. This is because if $X$ is a compact topological group then there exists a left-invariant metric $d$ on $X$, that is a metric such that for every $g\in X$ the rotation by $g$, that is the map $x\mapsto gx$ denoted $R_g$ is an isometry ($d(x,y)=d(R_g(x),R_g(y))$ for all $x,y\in X$). It is now easy to see that every isometry must be $\Fbar$-continuous, hence it is uniquely ergodic if and only if it has a dense orbit.  We also see that for an isometry every point is a generic point for an ergodic measure and has a uniquely ergodic orbit closure, which must be a minimal subset, and all isometries must have a closed set of ergodic measures.
\end{rem}

The concept of mean equicontinuity was introduced in \cite{LTY15}. Mean equicontinuity is equivalent to the mean-L-stable property introduced in \cite{F51}.
It is clear that if  $(X,T)$ is mean equicontinuous
then it is also $\Fbar$-continuous.

We have the following characterisation of mean equicontinuity. Note that according to Theorem \ref{thm-Fbar-continuous} we can replace the condition \ref{mean-2} of  Theorem~\ref{thm:mean-eq} by any condition listed in Theorem \ref{thm-Fbar-continuous}. For the reasons explained in Remark \ref{rem:FGL} below we mention only one such possibility.
\begin{thm}\label{thm:mean-eq}
For a topological dynamical system $(X,T)$ the following conditions are equivalent:
\begin{enumerate}[label=(\arabic*),ref=(\arabic*)]
    \item \label{mean-1} $(X,T)$ is mean equicontinuous;
    \item \label{mean-2} $(X\times X,T\times T)$ is $\Fbar$-continuous;
    \item \label{mean-2A}
    for every $(x,y)\in X\times X$ the system $\overline{\orb((x,y),T\times T)}$ is uniquely ergodic and the map $(x,y)\mapsto\mu(x,y)$ is continuous;
    \item \label{mean-3} $(X,T)$ is Weyl mean equicontinuous: for every $\eps>0$
there exists a $\delta>0$ such that for every $x,y\in X$
with $d(x,y)<\delta$ one has
\[
 \limsup_{n-m\to\infty} \frac{1}{n-m}\sum_{k=m}^{n-1}d(T^k(x),T^k(y))<\eps.
\]
\end{enumerate}
\end{thm}
\begin{proof}
\Implies{mean-1}{mean-2}:
As $(X\times X,T\times T)$ is also mean equicontinuous,
$(X\times X,T\times T)$ is $\Fbar$-continuous.

\IFF{mean-2}{mean-2A}: See Theorem \ref{thm-Fbar-continuous} \IFF{Fbar-equi-1}{Fbar-equi-2} and Remark \ref{rem:H-cont}.

\Implies{mean-2}{mean-3}:  As the metric $d(\cdot,\cdot)$
is a continuous function on $X\times X$, by
Theorem \ref{thm-Fbar-continuous}\ref{Fbar-equi-6} the function
\[
(x,y)\mapsto
\lim_{n-m\to\infty} \frac{1}{n-m}\sum_{k=m}^{n-1}d(T^k(x),T^k(y))
\]
is continuous.
Note that the limit in the above formula exists.

\Implies{mean-3}{mean-1}:
This is clear.
\end{proof}

\begin{rem} \label{rem:FGL} Note the equivalence \IFF{mean-1}{mean-2A} in Theorem \ref{thm:mean-eq} is stated as Theorem 4.2 in \cite{FGL18}. Note that the authors \cite{FGL18} worked in a much greater generality and obtained their result for mean equicontinuous actions of locally compact $\sigma$-compact amenable groups. Our proofs work for all countable Abelian %amenable
groups (see Remark \ref{rem:amenable}).
%shows that
%a topological dynamical system is mean equicontinuous
%if and only if in the product system $(X\times X,T\times T)$, for every $(x,y)\in X\times X$, the set $\V((x,y))$ is a singleton and the map $(x,y)\mapsto \V((x,y))$ is continuous.
%Combining Theorems \ref{thm-Fbar-continuous} and
%\ref{thm:mean-eq}, we obtain a new proof of this result.
\end{rem}

\begin{rem}
The property in Theorem \ref{thm:mean-eq}\ref{mean-3}
was called Banach mean equicontinuity in \cite{LTY15}. The authors of
%It is question in
\cite{LTY15} also asked whether Banach mean equicontinuity is equivalent to mean equicontinuity.
It is shown in  \cite{DG16} that the answer is positive for minimal systems and in \cite{QZ20} for general systems.
Our Theorem \ref{thm:mean-eq} provides a new proof of this fact for topological dynamical systems, in particular for  $\mathbb{Z}$-actions. In \cite{FGL18}, the authors examined  Weyl and Besicovitch equicontinuity for actions of locally compact $\sigma$-compact amenable groups, and proved that these notions are equivalent assuming that the action has a fully supported invariant measure or that the acting group is Abelian. Our approach generalises to any countable Abelian group action (see Remark \ref{rem:amenable}) without any further assumptions.
\end{rem}

We wrap up with yet another definition motivated by \cite{HLTXY21}. The authors of \cite{HLTXY21} introduced \emph{equicontinuity in the mean} and showed that it is equivalent to mean equicontinuity for minimal systems and the authors of \cite{QZ20} showed the equivalence for all topological dynamical systems. In analogy to these results, we introduce $\{\Fbar_n\}$-equicontinuity and show that it is actually equivalent to $\Fbar$-continuity.

\begin{defn}
	Let $(X,T)$ be a topological dynamical system. We say that $(X,T)$ is $\{\Fbar_n\}$-equicontinuous if for any $\epsilon>0$ there exists a $\delta>0$, such that for every $x,y\in X$ with
	$d(x,y)<\delta$ we have $\Fbar_n(x,y)<\eps$ for every $n\in\mathbb{N}$.
\end{defn}
It is easy to see that $\{\Fbar_n\}$-equicontinuity implies $\Fbar$-continuity. We will show that the converse also holds.

\begin{thm}\label{thm:fn-equi}
For a topological dynamical system $(X,T)$  the following conditions are equivalent:
  \begin{enumerate}[label=(\arabic*),ref=(\arabic*)]
        \item \label{last-1} $(X,T)$ is $\Fbar$-continuous;
        \item \label{last-2} $(X,T)$ is $\{\Fbar_n\}$-equicontinuous.
    \end{enumerate}
\end{thm}
\begin{proof}

\Implies{last-1}{last-2}:
Fix $\eps>0$. Use Corollary \ref{fn} to find $\eta > 0$ such that for every  $n\in\mathbb{N}$ and $x,y\in X$ with $\rho(\m(x,n),\m(y,n))<\eta$ we have $\Fbar_n(x,y)<\eps$. By Theorem \ref{thm-Fbar-continuous}\ref{Fbar-equi-1A} there exists $\delta>0$, such that for every $x,y\in X$ with $d(x,y)<\delta$ we have $\rho(\m(x,n),\m(y,n))<\eta$ for every  $n\in\mathbb{N}$. It is now clear that for every $x,y\in X$, the inequality $d(x,y)<\delta$ implies $\Fbar_n(x,y)<\eps$  for every  $n\in\mathbb{N}$.

\Implies{last-2}{last-1}: It is clear.
\end{proof}

\begin{rem}\label{rem:amenable}
Our approach works verbatim for actions of countable Abelian groups.
\end{rem}

\begin{rem}[Added in proof] When the present paper was being completed, two related papers: one authored by
Downarowicz and Weiss \cite{DW21}, the other by Xu and Zheng \cite{XZ21} appeared on the arXiv server. Both papers cotain results  overlapping  non-trivially with our work, although the methods of proof differ. We indicate the similarities here and we have also added some remarks in the main text of our paper. We believe that our new results   (Theorem \ref{theorem:1} and Corollary \ref{cor:fbar-0}, Corollary \ref{fn}, Theorem \ref{thm:fn-equi}, extension of Theorem \ref{thm-Fbar-continuous} by extra equivalent conditions \ref{Fbar-equi-1B} and \ref{Fbar-equi-1A})  are still interesting additions to the theory presented in \cite{DW21, FGL18, XZ21, ZZ20}.

In \cite{DW21}, the authors studied, among the others, the notion we call $\Fbar$-continuity, but they did not consider the pseudo-metric $\Fbar$. Instead, they studied:
\begin{enumerate}\item \emph{continuously pointwise ergodic} topological dynamical systems defined using condition appearing in Theorem \ref{thm-Fbar-continuous}\ref{Fbar-equi-3} (note that this condition was studied earlier in \cite{FGL18} without giving it a name);
\item \emph{uniform} topological dynamical systems defined by the condition appearing in Theorem \ref{thm-Fbar-continuous}\ref{Fbar-equi-7}.
\end{enumerate}
Downarowicz and Weiss also proved in \cite[Theorem 4.9]{DW21} that a topological dynamical system is uniform if and only if it is  continuously pointwise ergodic, which is the same as the equivalence \IFF{Fbar-equi-3}{Fbar-equi-7} in Theorem \ref{thm-Fbar-continuous} above. The proof presented in \cite{DW21} is different than ours.

In \cite{XZ21} the authors examined the pseudo-metric $\Fbar$ for countable amenable group actions. They repeated the main results of \cite{ZZ20} in that more general setting and recovered the characterisation of mean equicontinuity presented in Theorem \ref{thm:mean-eq}. Recall that this characterisation was  obtained earlier in \cite{FGL18} in an even more general setting (without the assumption that the acting group is countable). The authors of \cite{XZ21} have also noted the equality \eqref{eq:En=Wasser} in a more general setting of finite orbit segments for actions of countable groups.
\end{rem}

\section*{Acknowledgements} The present form of our results in Section 3 benefited greatly from our discussions with Tomasz Downarowicz. These discussions were propelled by the insightful remarks of  the anonymous reviewer, who noted that some assumptions we had made in the earlier version could be relaxed. We would like to thank both, Tomasz Downarowicz and the anonymous referee,  for their help. Although our old results in Section 3 were less complete than their present form, they were still sufficient to carry on all the proofs in Section 4.  Furthermore, we would like to thank the anonymous referee for careful reading and other helpful suggestions, that helped us to improve our paper.
J. Li  was partially supported by NNSF of China (12171298) and NSF of Guangdong Province (2018B030306024). D.~Kwietniak was supported by the National Science Centre, Poland, grant no. 2018/29/B/ST1/01340. The research of H.~Pourmand leading to the present publication has received funding from the Norwegian Financial
Mechanism 2014-2021 via the National Science Centre, POLS grant no.
2020/37/K/ST1/02770.
We are grateful to Leiye Xu and Liqi Zheng for sharing their work \cite{XZ21} with us.

\end{document}